\documentclass[12pt,letterpaper]{article}
\usepackage[parfill]{parskip}
\usepackage{amssymb}
\usepackage{amsmath}
\usepackage{amsthm}
\usepackage{amsfonts}

\usepackage[T1]{fontenc}
\usepackage{palatino}
\usepackage[text={6.5in, 9in}, centering]{geometry}

\usepackage[usenames]{color}
\newtheorem{theorem}{Theorem}
\newtheorem{definition}[theorem]{Definition}
\newtheorem{lemma}[theorem]{Lemma}

\newtheorem{remark}[theorem]{Remark}

\newcommand{\R}{\mathbb{R}}

\newcommand{\Z}{\mathbb{Z}}
\newcommand{\N}{\mathbb{N}}

\newcommand{\HB}{\dot B_{\infty,\infty}^{-\epsilon}}

\newcommand{\epsilonB}{\dot B_{\infty,\infty}^{-\epsilon}}

\begin{document}

\title{Frequency localized regularity criteria for the 3D Navier-Stokes equations}
\author{Z. Bradshaw and Z. Gruji\'c}

\maketitle

\begin{abstract}Two regularity criteria are established to highlight which Littlewood-Paley frequencies play an essential role in possible singularity formation in a Leray-Hopf weak solution to the Navier-Stokes equations in three spatial dimensions.  One of these is a frequency localized refinement of known Ladyzhenskaya-Prodi-Serrin-type regularity criteria restricted to a finite window of frequencies the lower bound of which diverges to $+\infty$ as $t$ approaches an initial singular time.
\end{abstract}

\section{Introduction}

The Navier-Stokes equations governing the evolution of a viscous, incompressible flow's velocity field $u$ in $\R^3\times (0,T)$ read
\[
\tag{3D NSE}
\begin{array}{ll}
	\partial_t u + u \cdot \nabla u
		= 
			  - \nabla p + \nu \Delta u + f
	& \mbox{~in~}\R^3\times(0,T) 
\\
	\nabla\cdot u
		=0 
	& \mbox{~in~}\R^3\times(0,T),
\end{array}
\]
where $\nu$ is the viscosity coefficient, $p$ is the pressure, and $f$ is the forcing. For convenience we take $f$ to be zero and set $\nu =1$.  The flow evolves from an initial vector field $u_0$ taken in an appropriate function space.

The regularity of Leray-Hopf weak solutions (i.e.~distributional solutions for $u_0\in L^2$ that satisfy the global energy inequality and belong to $L^\infty(0,T;L^2)\cap L^2(0,T;H^1)$ for any $T>0$) remains an open problem.  The best results available rely on critical quantities being finite, that is quantities which are invariant given the natural scaling associated with the Navier-Stokes equations.
In this note we provide several regularity criteria which highlight the essential role of high frequencies in a possibly singular Leray-Hopf weak solution.

Frequencies are interpreted in the Littlewood-Paley sense.  Let $\lambda_j=2^j$ for $j\in \Z$ be measured in inverse length scales and let $B_r$ denote the ball of radius $r$ centered at the origin.  Fix a non-negative, radial cut-off function $\chi\in C_0^\infty(B_{1})$ so that $\chi(\xi)=1$ for all $\xi\in B_{1/2}$. Let $\phi(\xi)=\chi(\lambda_1^{-1}\xi)-\chi(\xi)$ and $\phi_j(\xi)=\phi(\lambda_j^{-1})(\xi)$.  Suppose that $u$ is a vector field of tempered distributions and let $\Delta_j u=\mathcal F^{-1}\phi_j*u$ for $j\in \N $ and $\Delta_{-1}=\mathcal F^{-1}\chi*u$. Then, $u$ can be written as\[u=\sum_{j\geq -1}\Delta_j u.\]
If $\mathcal F^{-1}\phi_j*u\to 0$ as $j\to -\infty$ in the space of tempered distributions, then for $j\in \Z$ we define $\dot \Delta_j u = \mathcal F^{-1}\phi_j*u$ and have
\[u=\sum_{j\in \Z}\dot \Delta_j u.\]
For $s\in \R$, $1\leq p,q\leq \infty$ the homogeneous Besov spaces include tempered distributions modulo polynomials for which the norm
\begin{align*}
&\|u\|_{\dot B^s_{p,q}}:= 
\begin{cases} 
 \bigg(\sum_{ j\in \Z} \big(    \lambda_j^s \|\dot \Delta_j u \|_{L^p(\R^n)}  \big)^q \bigg)^{1/q}   & \text{ if } q<\infty  
\\ \sup_{j\in \Z} \lambda_j^s \|\dot \Delta_j u \|_{L^p(\R^n)} & \text{ if } q=\infty
\end{cases}, 
\end{align*}is finite. See \cite{BaChDa--Book} for more details.

Given a Leray-Hopf weak solution $u$ that belongs to $C(0,T;\HB)$ 
for some $\epsilon$ in $(0,1)$,
we define the following upper and lower endpoint frequencies: for $t$ in $(0,T)$ let
 \begin{align}\label{Jhigh}
 	J_{high}(t)=\log_2\bigg[ c_1 \|u(t)\|_{\epsilonB}^{1/(1-\epsilon)}  \bigg]			,
 \end{align}
 and
 \begin{align}\label{Jlow}
 	J_{low}(t) =	\log_2 \bigg[  \bigg(c_2 \frac {\|u(t)\|_{\HB}} { \|u\|_{L^\infty(0,T;L^2)}  }  \bigg)^{2/(3-2\epsilon)} \bigg],
 \end{align}
 where $c_1$ and $c_2$ are universal constants (their values will become clear in Section 2).
Our first regularity criterion shows  $J_{low}$ and $J_{high}$ determine the Littlewood-Paley frequencies which, if well behaved at a finite number of times prior to a possible blow-up time, prevent singularity formation.

\medskip

\begin{theorem}
\label{thrm:1}
Fix $\epsilon\in (0,1)$ and $T>0$, and assume that $u\in C(0,T;\HB)$ is a Leray-Hopf weak solution to 
3D NSE on $[0,T]$.  
If there exists $t_0\in (0,T)$ such that
\begin{align}\label{cond:modes}
	\sup_{J_{low}(t_0)\leq j\leq J_{high}(t_0)} \lambda_j^{-\epsilon}\| \dot\Delta_j u (t_i) \|_{L^\infty} 
	\leq
	 	\|u(t_0) \|_{\HB}, 
\end{align}
where $\{t_i\}_{i=1}^k \subset (t_0,T)$ is a finite collection of $k$ times satisfying
 \[ t_{i+1}-t_i>  \bigg(\frac {c_3} {\|u(t_0)\|_{\dot B_{\infty,\infty}^{-\epsilon}}}\bigg)^{2/(1-\epsilon)}\qquad ( i=0,\ldots,k-1),\]
 and
 \[
 T-t_k < \bigg(\frac {2c_3} {\|u(t_0)\|_{\dot B_{\infty,\infty}^{-\epsilon}}}\bigg)^{2/(1-\epsilon)}
 \]
for a universal constant  $c_3$, then $u$ can be smoothly extended beyond time $T$.
\end{theorem}

The novelty here is that the solution remains finite provided only a finite range of frequencies remain subdued at a finite number of  uniformly spaced times.
If $u$ is not in the energy class then a partial result can be formulated since $J_{high}$ does not depend on $\|u\|_{L^\infty(0,T;L^2)}$.  In particular, we just need to replace \eqref{cond:modes} with 
\[\sup_{ j\leq J_{high}(t)} \lambda_j^{-\epsilon}\| \dot\Delta_j u (t_i) \|_{L^\infty} 
	\leq
	 	\|u(t) \|_{\HB},\]
and assume $u$ is the mild solution for $u_0\in \HB$ which is a strong solution on $[0,T)$ (note that a local-in-time existence theory for mild solution is available in $\HB$).
\bigskip 

Our second result is a refinement of a well known class of regularity criteria (see, e.g.,  \cite{LR--Book}): if $u$ is a Leray-Hopf weak solution to 3D NSE on $\R^3\times [0,T]$ satisfying
\[
\int_0^T  \|u  \|_{L^p}^q\,dt <\infty ,
\]
for pairs $(p,q)$ where $3\leq p\leq \infty$, $2\leq q\leq \infty$, and
\[
\frac 2  q +\frac 3 p =1,
\]
then $u$ is smooth.  This is the Ladyzhenskaya-Prodi-Serrin class for non-endpoint values of $(p,q)$. The case $p=\infty$ is the Beale-Kato-Majda regularity criteria.  The case $p=3$ was only (relatively) recently proven in \cite{ISS}.  Similar criteria can be formulated for a variety of spaces larger than $L^p$ when $p>3$.  For example, Cheskidov and Shvydkoy give the following Ladyzhenskaya-Prodi-Serrin-type regularity criteria in Besov spaces (see \cite{ChSh2}): if $u$ is a Leray-Hopf solution and $u\in L^{2/(1-\epsilon)}(0,T;\HB)$, then $u$ is regular on $(0,T]$.  A regularity criterion for weakly time integrable Besov norms in critical classes appears in \cite{Bae}.   In the endpoint case when $\epsilon=-1$, smallness is needed either over all frequencies (see \cite{ChSh2}) or over high frequencies provided a Beale-Kato-Majda-type bound holds for the projection onto low frequencies (see \cite{ChSh}).   Our result is essentially a refinement of the non-endpoint regularity criteria given in \cite{ChSh2}.

\medskip

\begin{theorem}\label{thrm:LPS1}
Fix $\epsilon\in (0,1)$ and $T>0$, and assume that $u\in C(0,T;\HB)$ is a Leray-Hopf weak solution to 3D NSE on $[0,T]$. 
If 
\[
\int_0^T  \bigg(\sup_{J_{low}(t)\leq j \leq J_{high}(t)} \lambda_j^{-\epsilon}\|\dot\Delta_{j} u(t)\|_{\infty}\bigg)^{2/(1-\epsilon )}\,dt < \infty,
\] 
then $u$ is regular on $(0,T]$.
\end{theorem}

Clearly $J_{high}$ blows up more rapidly than $J_{low}$ as $t\to T^-$ and therefore an increasing number of frequencies are relevant as we approach the possible blow-up time.  
It is unlikely that this can be improved for weak solutions in supercritical classes like Leray-Hopf solutions. On one hand, the upper cutoff is available because of local well-posedness for the subcritical quantity $\|u(t)\|_{\HB}$ which suppresses high frequencies at times close to and after $t$.  On the other hand, the supercritical quantity $\|u\|_{L^\infty(0,T;L^2)}$ plays a crucial role in suppressing low frequencies. Any supercritical quantity is sufficient; for example, if we replace $L^\infty L^2$ with $L^\infty L^p$ for some $2<p<3$, then the lower cutoff function is
\[ 
 	J_{low}(t)=	\log_2 \bigg[  \bigg( \frac {\|u(t)\|_{\HB}} {c \|u\|_{L^\infty(0,T;L^p)}  }  \bigg)^{p/(3-p\epsilon)} \bigg].	
 \]
Note that $p/(3-p\epsilon)= 1/(1-\epsilon)$ only when $p=3$, i.e.~the exponents in the cutoffs will match
only when we reach a critical class $L^\infty(0,T; L^3)$.

\section{Technical lemmas}

Local existence of strong solutions for data in the subcritical space $\HB$ is known, see \cite{LR--Book}.  Results in spaces close to $\HB$ are given in \cite{KoOgTa,Sawada}.  Indeed, the proof of \cite[Theorem 1]{KoOgTa} can be modified to show that if $a\in \HB$, then the Navier-Stokes equations have a unique strong solution $u$ which persists at least until time 
\begin{align} T_*= \bigg(\frac {c_0} {\|a\|_{\dot B_{\infty,\infty}^{-\epsilon}}}\bigg)^{2/(1-\epsilon)},\label{def:T_*}
\end{align}
for a universal constant $c_0$.  Moreover we have 
\begin{align}\label{ineq.u}
\|u(t)\|_{\HB} \leq c_0 \|a\|_{\HB},
\end{align}
and 
\begin{align}\label{ineq.gradu}
t^{1/2} \|\nabla u(t)\|_{\HB}  \leq c_0 \|a\|_{\HB},
\end{align} for any $t\in (0,T_*)$ (the value of $c_0$ changes from line to line but always represents a universal constant).  Since the proof of this is nearly identical to the proof of \cite[Theorem 1]{KoOgTa} it is omitted.  Note that by \cite[Proposition 3.2]{LR--Book}, the left hand side of \eqref{ineq.gradu} can be replaced by $t^{1/2}\|u\|_{\dot B ^{1-\epsilon}_{\infty,\infty}}$.

Given a solution $u$ and a time $t$ so that $u(t)\in \HB$, let $t'= t+T_*/2$ and $t''=t+T_*$ where $T_*$ is as in \eqref{def:T_*} with $a=u(t)$.  
We now state and prove several (short) technical lemmas.  

\medskip

\begin{lemma}\label{lemma.supercritical}  Fix $\epsilon\in [0,3/2)$ and $T>0$.
If $u$ is a Leray-Hopf weak solution to 3D NSE on $[0,T]$ and $u(t)\in \HB$ for some $t\in [0,T]$, then for any $M>0$ we have
\[
\lambda_j^{-\epsilon} \| \dot \Delta_j u(t) \|_{\infty} \leq M,
\]
provided \[ j\leq \log_2 \bigg[ \bigg( c \frac M {\|u\|_{L^\infty(0,T;L^2)}}  \bigg)^{2/(3-2\epsilon)} \bigg] \] 
for a suitable universal constant $c$.
\end{lemma}
\begin{proof}
Assume $u$ is a Leray-Hopf weak solution on $[0,T]$ and $t\in [0,T]$ such that $\|u(t)\|_{\HB}<\infty$. 
By Bernstein's inequalities we have 
\[
\|\dot \Delta_j u(t)\|_\infty \leq \lambda_j^{3/2}\|\dot \Delta_j u(t)\|_{2}.
\]
Since $u\in L^\infty(0,T;L^2)= L^\infty(0,T;\dot B_{2,2}^0)$, for any $j\in \Z$,
\[
 \lambda_{j}^{-\epsilon} \|\dot \Delta_{j} u\|_\infty \leq c \lambda_j^{3/2-\epsilon }\| u\|_{L^\infty(0,T; L^2)}.
\]
Let 
\[
J(t) = \log_2 \bigg[  \bigg( \frac {M} {c \|u\|_{L^\infty(0,T;L^2)}  }  \bigg)^{2/(3-2\epsilon)} \bigg];
\] 
then
\[  
\sup_{j\leq J }\lambda_j^{-\epsilon} \|\dot \Delta_j u\|_\infty \leq M.
\]

\end{proof}

\begin{lemma}\label{lemma:modekiller}Fix $\epsilon\in (0,1)$ and $T>0$, and assume $u$ is a Leray-Hopf weak solution to 3D NSE on $[0,T]$ belonging to $C(0,T;\HB)$.  Then, for any $t_1\in (0,T)$ and all $t\in [t_1',t_1'']$ we have
\[
   \sup_{\{j\in \Z: j\leq J_{low}\text{ or }j\geq J_{high} \} } \| \dot \Delta_j u(t) \|_{L^\infty} \leq\frac 1 2 \|  u(t_1) \|_{\HB},
\]
where $J_{high}$ and $J_{low}$ are defined by \eqref{Jhigh} and \eqref{Jlow}.
\end{lemma}
\begin{proof}

Using subcritical local well-posedness in $\HB$ at $t_1$ we have that there exists a mild/strong solution $v$ defined on $[t_1,t_1'']$.  By \eqref{ineq.gradu} we have
\[
  (t-t_1)^{1/2}\|v(t)\|_{\dot B_{\infty,\infty}^{1-\epsilon}}\leq c_0 \|  v(t_1)\|_{\epsilonB}
\]
for all $t\in (t_1,t_1'')$.  Since $v(t_1)=u(t_1)\in L^2$ and since the strong solution $v$ is smooth, integration by parts
verifies that $v$ is also a Leray-Hopf weak solution to 3D NSE.  The weak-strong uniqueness result of \cite{May} 
then guarantees that $u=v$ on $[t_1,t_1'']$.  
Thus, for any $t\in [t_1',t_1'']$,
\[
 \lambda_j^{ -\epsilon} \| \dot \Delta_j u (t)\|_\infty \leq c \lambda_j^{-1} \|u(t_1)\|_{\epsilonB}^{1/(1-\epsilon)+1}
\]
for all $j\in \Z$. By \eqref{Jhigh}
we conclude that
\[
\sup_{j\geq J_{high}} \lambda_j^{ -\epsilon} \| \dot \Delta_j u (t)\|_\infty \leq \frac 1 2 \|u(t_1)\|_{\epsilonB}.
\] 

The low modes are eliminated using Lemma \ref{lemma.supercritical} with $M=\| u(t_1)\|_{\HB}/2$.
\end{proof}

\begin{definition}
We say that $t$ is an \emph{escape time} if there exists some $M > 0$ such that $t=\sup \{s\in (0,T): \|u(s) \|_{\HB}< M  \}$.
\end{definition}

\medskip

\begin{lemma}\label{lemma.escape}Fix $\epsilon\in (0,1)$ and $T>1$, and assume $u$ is a Leray-Hopf weak solution to 3D NSE on $[0,T]$ belonging to $C(0,T;\HB)$.
Let $\mathcal E$ denote the collection of escape times in $(0,T)$ and let $I=\cup_{t\in \mathcal E} (t',t'')$.
Then
\begin{align}\label{eq.infinity1}
\int_0^{T}  \|u(t)\|_{\dot B^{-\epsilon}_{\infty,\infty }}^{2/(1-\epsilon )}\,dt=\infty,
\end{align}
if and only if
\begin{align}\label{eq.infinity2}
\int_I  \|u(t)\|_{\dot B^{-\epsilon}_{\infty,\infty }}^{2/(1-\epsilon )}\,dt=\infty.
\end{align}
\end{lemma}
\begin{proof}
It is obvious that \eqref{eq.infinity2} implies \eqref{eq.infinity1}. 

Assume \eqref{eq.infinity1}.
Let $\{ t_k\}_{k\in \N}\subset (0,T)$ be an increasing sequence of escape times
 which converge to $T$ at $k\to \infty$. Clearly $\|u(t_k) \|_{\HB}$ blows up as $k\to\infty$.  
Since $u\in C(0,T;\HB)$, $\|u(t_{k_1})\|_{\HB}< \|u(t_{k_2})\|_{\HB}$ for all $k_1<k_2$.   

We have two cases depending on the condition
\begin{align}
\label{cond:cases}
\exists \, t_{k_0} \in \{t_k\} \mbox{~such that~} 
\forall\, k\geq k_0 \mbox{~we have~} t_{k+1}'\leq t_k''.
\end{align}

Case 1: If \eqref{cond:cases} is true, then $ [t_0',T)=\cup_{k\geq k_0} [t_k',t_k'').$  In this case let $I=[t_0',T)$.  Clearly
\[
\int_{I} \| u(t) \|_{\HB}^{2/(1-\epsilon)}\,dt =\infty.
\]
Case 2: If \eqref{cond:cases} is false then there exists an infinite sub-sequence of $\{ t_k \}$, which we label $\{ s_k \}$, such that $s_k''<s_{k+1}'$ for all $k\in \N$. In this case let $I=\cup_{k\in \N}[s_k',s_k'')$.  Then,
\[
\int_{I} \| u(t) \|_{\HB}^{2/(1-\epsilon)}\,dt \geq \sum_{k\in \N} \frac {T^*(s_k)} 2 \|u(s_k)  \|_{\HB}^{2/(1-\epsilon)}= \sum_{k\in \N} \frac {c_0^{2/(1-\epsilon)}} 2 =\infty.
\]
In either case, we have shown that \eqref{eq.infinity1} implies \eqref{eq.infinity2}.
\end{proof}

\section{Proofs of Theorem \ref{thrm:1} and Theorem \ref{thrm:LPS1}}

\begin{proof}
[Proof of Theorem \ref{thrm:1}] 

Fix $\epsilon\in (0,1)$ and $T>0$, and assume $u\in C(0,T;\HB)$ is a Leray-Hopf weak solution to 3D NSE on $[0,T]$. Assume $t_0,\ldots,t_k$ are as in the statement of the lemma.
It suffices to show \[\|u(t_k)\|_{\HB}\leq \|u(t_0)\|_{\HB},\]
since then we re-solve at $t_0$ and, by local-in-time well-posedness and the weak-strong uniqueness of \cite{May}, see that $u$ is regular at time $T$.

If $k=0$, then we are done.  Otherwise note that $t_1\in (t_0',t_0'')$. Apply Lemma \ref{lemma:modekiller} at $t_0$ to conclude that
\[\|u(t_1)\|_{\HB}\leq \|u(t_0)\|_{\HB}.\]
If $k=1$, then we are done.  Otherwise, we repeat the argument and eventually obtain
\[\|u(t_k)\|_{\HB}\leq \|u(t_0)\|_{\HB},\]
which completes the proof.
\end{proof}

\begin{proof}
[Proof of Theorem \ref{thrm:LPS1}] Assume $u$ is a Leray-Hopf weak solution on $[0,T]$ which belongs to $C(0,T;\HB)$. 

By Lemma \ref{lemma.supercritical} with $M=\| u(t)\|_{\HB}/2$ it follows that
\begin{align} \label{ineq.low} 
\sup_{j\leq J_{low}(t)}\lambda_j^{-\epsilon} \|\dot \Delta_j u(t)\|_\infty < \frac 1 2 \|u(t)\|_{\HB}.
\end{align}

If $u$ loses regularity at time $T$, local well-posedness in $\HB$ implies that
\[
	\|u(t)\|_{\HB}\geq \bigg(\frac {c_*} {T-t}  \bigg)^{(1-\epsilon)/2},
\]
for a small universal constant $c_*$. 
Therefore,
\[
\int_0^T \| u(t) \|_{\HB}^{2/(1-\epsilon)}\,dt =\infty.
\]

Let $\mathcal E$ denote the collection of escape times in $(0,T)$ and let $I=\cup_{t\in \mathcal E} (t',t'')$.  By Lemma \ref{lemma.escape}
\[
\int_{I} \| u(t) \|_{\HB}^{2/(1-\epsilon)}\,dt =\infty.
\]
For each $t\in I$ there exists an escape time $t_0(t)$ so that $t\in (t_0',t_0'')$. Thus,
\[
\frac 1 2 \bigg( \frac {c_0} { \| u(t_0) \|_{\HB}} \bigg)^{2/(1-\epsilon)} \leq  t-t_0 \leq \bigg( \frac {c_0} { \| u(t_0) \|_{\HB}} \bigg)^{2/(1-\epsilon)}.
\]
By re-solving at $t_0$ using subcritical well-posedness, inequality \eqref{ineq.gradu}, and weak-strong uniqueness (see \cite{May}), we have
\[
(t-t_0)^{1/2} \|u(t)\|_{\dot B^{1-\epsilon}_{\infty,\infty}  } \leq c_0 \|u(t_0)\|_{\HB}.
\] 
Consequently,
\[
\lambda_j^{-\epsilon} \|\dot \Delta_j u(t)\|_\infty\leq 2 c_0 \lambda_j^{-1} \|u(t_0)\|_{\HB}^{1+1/(1-\epsilon)}\leq 2 c_0 \lambda_j^{-1} \| u(t)\|_{\HB}^{1+1/(1-\epsilon)}, 
\]
where we have used the fact that $t_0$ is an escape time. Using \eqref{Jhigh}
we obtain
\begin{align}\label{ineq.high}
    \sup_{j\geq J_{high}(t)}\lambda_j^{-\epsilon} \| \dot \Delta_j u (t) \|_{\infty} < \frac {\|u(t)\|_{\HB}} 2.
\end{align}

Combining \eqref{ineq.low} and \eqref{ineq.high} yields
\[
\int_{I} \bigg(  \sup_{J_{low}(t)\leq j \leq J_{high}(t)} \lambda_j^{-\epsilon} \|\dot \Delta_j u(t)\|_{L^\infty}  \bigg)^{2/(1-\epsilon )} \,dt = \infty,
\] 
which proves Theorem \ref{thrm:LPS1}.   
\end{proof}

\begin{remark}
{\em 
If we only wanted to eliminate low frequencies in Theorem \ref{thrm:LPS1} then an alternative proof is available which we presently sketch.  Decompose $[0,T]$ into adjacent, disjoint intervals $[t_k,t_{k+1})$ with $t_{k+1}-t_k\sim 2^{-k}T$.  Then, a solution which is singular at $T$ must satisfy 
\[
2^k\lesssim \| u(t\sim t_k) \|_{\epsilonB}^{2/(1-\epsilon)}.
\]
Using the Bernstein inequalities we have
\begin{align*}
\int_{t_k}^{t_{k+1}} \bigg(\sup_{j\leq J_0(t)} \lambda_j^{-\epsilon} \| \dot \Delta_j u(t ) \|_{\infty}\bigg)^{2/(1-\epsilon)}\,dt
&\leq \int_{t_k}^{t_{k+1}} \bigg(\sup_{j\leq J_0(t)} \lambda_j^{3/2-\epsilon} \| \dot \Delta_j u(t ) \|_{2}\bigg)^{2/(1-\epsilon)}\,dt
\\&\lesssim \|u\|_{L^\infty L^2}^{2/(1-\epsilon )}\lambda_{J_0(t)}^{(3-2\epsilon)/(1-\epsilon)} (t_{k+1}-t_k)
\\&\lesssim \|u\|_{L^\infty L^2}^{2/(1-\epsilon )}2^{J_0(t)(3-2\epsilon)/(1-\epsilon)} 2^{-k}.
\end{align*}
Define $J_0$ so that $J_0(t) (3-2\epsilon)/(1-\epsilon) = k/2$ for $t\in [t_k,t_k+1)$. Then, terms on the right hand side are summable and we obtain
\[
\int_0^T \bigg(\sup_{j\leq J_0} \lambda_j^{-\epsilon} \| \dot \Delta_j u(t ) \|_{\infty}\bigg)^{2/(1-\epsilon)}\,dt <\infty.
\]
Since the integral over all modes must be infinite at a first singular time, we conclude
\[
\int_0^T \bigg(\sup_{j\geq J_0(t)} \lambda_j^{-\epsilon} \| \dot \Delta_j u(t ) \|_{\infty}\bigg)^{2/(1-\epsilon)}\,dt =\infty.
\]
Further analyzing the definition of $J_0$ and the lower-bound for the $\epsilonB$ norm we see that
\[
J_0(t)\sim\log_2\bigg( \|u(t)\|_{\epsilonB}^{2/(3-2\epsilon)}  \bigg),
\]
which matches the rate found using the other approach.  
}
\end{remark}

\subsection*{Acknowledgements.}

The authors are grateful to V.~\v Sver\'ak for his insightful comments which simplified the proofs.

Z.\,G. acknowledges
support of the \emph{Research Council of Norway} via the grant
213474/F20 and the \emph{National Science Foundation} via the grant
DMS 1212023.

\bibliographystyle{plain}
\bibliography{references}

\end{document}